\documentclass[a4paper,12pt,reqno]{amsart}

\usepackage{amsmath}        
\usepackage{amsfonts}       
\usepackage{amsthm}         
\usepackage{amssymb}%
\usepackage{scalerel}
\usepackage{verbatim} 

\numberwithin{equation}{section}

\theoremstyle{plain}
\newtheorem{theorem}{Theorem}[section]
\newtheorem{lemma}[theorem]{Lemma}
\newtheorem{proposition}[theorem]{Proposition}
\newtheorem{corollary}[theorem]{Corollary}

\newtheorem*{mainthm}{Main Theorem}

\theoremstyle{definition}
\newtheorem{definition}[theorem]{Definition}
\newtheorem{remark}[theorem]{Remark}

\DeclareMathOperator{\rad}{rad}

\DeclareMathOperator{\chr}{char}
\DeclareMathOperator{\ndeg}{ndeg}

\DeclareMathOperator{\ql}{ql}
\DeclareMathOperator{\gr}{Gr}
\DeclareMathOperator{\diag}{diag}
\DeclareMathOperator*{\bp}{\scalerel*{\perp}{\sum}}
\DeclareMathOperator{\typ}{type}

\newcommand{\BB}{\mathcal{B}}

\newcommand{\HH}{\mathbb{H}}
\renewcommand{\phi}{\varphi}

\newcommand{\dom}{\preccurlyeq} 

\newcommand{\iw}{\mathrm{i}_{\mathrm{W}}} 
\newcommand{\iql}{\mathrm{i}_{\mathrm{d}}} 
\newcommand{\itt}{\mathrm{i}_{\mathrm{t}}} 

\newcommand{\an}{\mathrm{an}}

\newcommand{\com}{\mathrm{c}}

\newcommand{\qf}[1]{\langle #1 \rangle} 
\newcommand{\pff}[1]{\langle\!\langle #1 \rangle\!\rangle} 
\newcommand{\qpf}[1]{\langle\!\langle #1 ]]} 

\newcommand{\implylr}{\Longleftrightarrow}

\newcommand{\sml}{\stackrel{{\rm s}}{\sim}}

\newcommand{\vis}{\stackrel{{\rm v}}{\sim}}
\newcommand{\wit}{\stackrel{{\rm w}}{\sim}}

\begin{document}
\title[Nonsimilar half neighbors]{Nonsimilar half-neighbors over fields of characteristic $2$}

\author{Detlev W. Hoffmann}
\address{Fakult\"at f\"ur Mathematik, Technische Universit\"at Dortmund, D-44221 Dortmund, Germany}
\email{detlev.hoffmann@math.tu-dortmund.de}

\author{Magnus Wiedeking}
\email{magnus.wiedeking@tu-dortmund.de}
\date{\today}

\begin{abstract}
The total isotropy index of a quadratic form $\phi$ over a field $F$ is the maximum dimension
of any totally isotropic subspace of $\phi$. If $\phi$ is anisotropic and  $\psi$ is another anisotropic quadratic form over $F$ of
the same dimension, then $\phi$ and $\psi$ are called Vishik-equivalent if, over any field 
extension $E/F$, their total isotropy indices are the same.  In characteristic $\neq 2$,
Vishik-equivalence implies similarity in all dimensions $\leq 7$ and in all odd dimensions,
but there are counterexamples in all even dimensions $\geq 8$.
In this paper, 
we construct semi-singular anisotropic 
quadratic forms of dimension $2^m$ for any $m\geq 3$ and defined over a suitable extension of
any given field $F_0$ of characteristic $2$ that are Vishik-equivalent but not similar, thus
completing the list of such examples provided earlier by the
first author and Krist\'yna Zemkov\'a.
\end{abstract}

\maketitle

\section{Introduction} \label{Sec:Intro}

Consider two quadratic forms $\phi$ and $\psi$ over a field $F$ and of the same finite dimension
such that their associated projective
quadrics $X_\phi$ and $X_\psi$ are smooth, i.e., such that the (bilinear) radicals of these
forms have dimension at most $1$ and the restriction to the radical is anisotropic.
Such forms are called nondegenerate, and
radicals of dimension $1$ can then only occur in characteristic $2$
for odd-dimensional forms.  Now these two forms are similar iff their quadrics are
isomorphic as varieties (see \cite[Cor.~69.5]{ekm}).  In this situation, 
if $E/F$ is any field extension, then obviously, the total isotropy index of $\phi$ and
$\psi$ after scalar extension to $E$ (i.e., the dimensions
of any maximal totally isotropic subspace of $\phi_E$ resp. $\psi_E$) will be the same.
Now  a theorem first proved by Vishik \cite{v} in characteristic $0$,
later by Karpenko \cite{k} in characteristic not $2$
(for a characteristic free version, see \cite[Theorem~93.1]{ekm}), states that the
latter condition on the equality of the total isotropy indices over any field extension is equivalent to the fact
that $X_\phi$ and $X_\psi$ have isomorphic motives in the category of
(integral or $\bmod\,2$) Chow motives, in which case we call $\phi$ and $\psi$
motivically equivalent.  To study the relation between similarity and motivic equivalence,
one may restrict to the study of anisotropic forms, and so we may ask: Are
two motivically equivalent anisotropic quadratic forms similar?

\renewcommand*{\thefootnote}{$\dagger$}
In characteristic not $2$, it was shown by Izhboldin \cite{i1}, \cite{i2} that the answer is positive
for forms of dimension $\leq 7$ as well as for forms of odd dimension, but that
there are counter-examples
in all even dimensions $\geq 8$ except possibly $12$.\footnote{While this paper was written,
Nils Poprawski constructed the first examples of nonsimilar motivically equivalent quadratic forms
in dimension $12$.}  In characteristic $2$, the answer is still positive
for odd dimensional nondegenerate forms,
 see \cite[Theorem~27.3]{ekm}.
The proofs actually use the criterion on the equality of the total isotropy indices
over any field extension.  So we give this criterion a name and
call two (anisotropic) quadratic forms $\phi$ and $\psi$  over $F$ of the same dimension
Vishik-equivalent if, over any field extension $E/F$, the total isotropy indices of
$\phi_E$ and $\psi_E$ are the same.  In this paper and in view of Izhboldin's result,
we focus on fields of characteristic $2$.  Then it is very well possible that
anisotropic quadratic forms may be degenerate, so we systematically include this case
and we ask in all generality:  Are two Vishik-equivalent anisotropic quadratic forms similar? 

Recall that over a field $F$ of characteristic $2$, a quadratic form $\phi$ can be decomposed
as $\phi\cong \phi_r\perp\phi_s$ with $\phi_r$ nonsingular of even dimension
$2r$, and $\phi_s$ totally singular of dimension $s$,
where $\phi_s$ is just the restriction of $\phi$
to its bilinear radical, also called the quasi-linear part of $\phi$ and denoted by $\ql(\phi)$.
So $\dim(\phi)=2r+s$, and $(r,s)$ is called the type of $\phi$, and we write
$\typ(\phi)=(r,s)$ (for more details, see the next
section).  Note that if two anisotropic forms $\phi$ and $\psi$ over $F$ are Vishik-equivalent,
then $\typ(\phi)=\typ(\psi)$  (see \cite[Lemma~2.8]{hz}).

Still in characteristic $2$, Vishik-equivalence of anisotropic forms
of type $(r,s)$ and dimension $d=2r+s$ implies similarity
in the following cases (the list is taken from  \cite[Theorem 1.1]{hz}):
\begin{itemize}
\item $d\leq 5$ (\cite[Theorem~4.2]{hz} and further references there);
\item $s=1$ (\cite[Theorem~27.3]{ekm});
\item $s=3$ (\cite[Theorem~3.1]{hz});
\item one of the forms is a Pfister neighbor of codimension $\leq 5$
(\cite[Corollary~5.4]{hz});
\item one of the forms is a nondegenerate so-called excellent form (\cite[Corollary~5.5]{hz}).
\end{itemize}
In \cite{hz}, using the theory of half-neighbors, constructions were given
to obtain fields of characteristic $2$ and anisotropic quadratic forms $\phi$ and
$\psi$ over $F$ of type $(r,s)$ and dimension $d=2r+s$ that are Vishik-equivalent but not
similar for the following types:
\begin{itemize}
\item $d=2^n$ and type $(2^{n-1},0)$, $n\geq 3$  (\cite[Theorem 6.5, Corollary 6.6]{hz});
\item $d=2^3$ and types  $(3,2)$, $(2,4)$, $(1,6)$ (\cite[\S 6.3]{hz}).
\end{itemize}
In this paper, we complete the list of such examples in dimensions $2^n\geq 8$ that
are not totally singular (i.e., of type $(r,s)$ with $r\geq 1$).  The main result of this paper is the following.
\begin{mainthm}  To any nonnegative integers $n,r,s$ with $n\geq 3$, $r\geq 1$
and $2^n=2r+s$, and for any field $F_0$ of characteristic $2$,  there exist a field
extension $F/F_0$ and
anisotropic quadratic forms $\phi$ and $\psi$ over $F$ of type $(r,s)$ such that $\phi$
and $\psi$ are Vishik-equivalent but not similar.
\end{mainthm}
Our constructions will also invoke the theory of half-neighbors.  A crucial ingredient
will be the study of similarity factors of totally singular forms (i.e., forms
of type $(0,s)$), and in particular the construction of totally singular forms $\sigma$
whose group of similarity factors $G_F(\sigma)$ is trivial, i.e., 
$G_F(\sigma)=F^{*2}$. 

In the next section, we recall basic facts about quadratic forms in characteristic~$2$.
In \S\,3, we introduce Pfister forms and half-neighbors and prove some
facts that will be needed in our constructions.
In \S\,4 we recall some basic properties of totally singular forms and their similarity factors.
Totally singular forms whose groups of similarity factors are trivial will play a central
role in our constructions.
In the final section 5, we present our constructions of nonsimilar half-neighbors which
will thus yield our examples of nonsimilar Vishik-equivalent forms and complete the 
proof of the Main Theorem.

\section{Basic definitions and facts}\label{Sec:Def}
For all undefined terminology and 
facts concerning the theory of quadratic forms that are mentioned without further reference
(in particular
in characteristic $2$), we refer to \cite{ekm}, \cite{hl}, \cite{h}.

Let $\phi$ be a quadratic form (or form, for short) defined on a finite-dimensional vector space
$V$ over a field $F$. The dimension of $\phi$ is defined to be $\dim(\phi)=\dim(V)$.
The symmetric bilinear form associated with $\phi$ is given by
$$b_\phi(x,y)=\phi(x+y)-\phi(x)-\phi(y)\quad (\forall x,y\in V).$$
We have the usual notions of isometry $\phi\cong\psi$ and orthogonal sum $\phi\perp\psi$
of two forms $\phi$ and $\psi$.
$\phi$ is said to be isotropic if there exists an $x\in V\setminus\{ 0\}$ with
$\phi(x)=0$, otherwise $\phi$ is said to be anisotropic. 

The (bilinear) radical of $\phi$ is defined to be
$\rad(\phi)=\rad(b_\phi)=\{ x\in V\,|\,b_\phi(x,y)=0\ \forall y\in V\}$.  
If $\chr(F)\neq 2$,
the restriction of $\phi$ to its radical is the zero form (i.e., $\phi(x)=0$ for all
$x\in\rad(\phi)$).
This need not be the case in characteristic $2$, so in our study  we systematically
include forms that may have a radical of dimension $>0$.  We call $\phi$ nonsingular if
$\dim\rad(\phi)=0$.

 The value set of
$\phi$ is defined to be $D_F(\phi)=\{ \phi(x)\,|\,x\in V\}\cap F^*$ and the group of
similarity factors of $\phi$ is given by $G_F(\phi)=\{ a\in F^*\,|\,\phi\cong a\phi\}$.
Note that $F^{*2}\leq G_F(\phi)$.
If $1\in D_F(\phi)$, then $G_F(\phi)\subseteq D_F(\phi)$, and
$\phi$ is said to be round if $D_F(\phi)=G_F(\phi)$.
For later purposes, we also define $D_F^0(\phi)=D_F(\phi)\cup \{ 0\}$ and
$G_F^0(\phi)=G_F(\phi)\cup \{ 0\}$.

A subspace $W\subset V$ is said to be totally isotropic if $\phi(x)=0$ for all
$x\in W$.  Now a totally isotropic subspace is maximal with respect to inclusion
iff it is maximal with respect to dimension, and the dimension of such a maximal
totally isotropic subspace is called the total isotropy index of $\phi$ and
denoted by $\itt(\phi)$.

From now on, all fields will be of characteristic $2$.   Then nonsingular forms are always
even-dimensional and can be written (up to isometry) as orthogonal sum of
two-dimensional nonsingular forms of shape $[a,b]$, the form given by the 
polynomial $ax^2+xy+by^2$, $a,b\in F$. The form $\HH =[0,0]\cong [0,a]$ (any $a\in F$)
is called hyperbolic plane,
it is the unique (up to isometry) nonsingular isotropic $2$-dimensional form.
A totally singular form of dimension $n\geq 1$ is a form of shape
$\qf{a_1,\ldots,a_n}$, $a_i\in F$,
given by the polynomial $a_1x_1^2+\ldots a_nx_n^2$ (in this case, any basis of the underlying
vector space will be an orthogonal basis).


A quadratic form $\phi$ of dimension $n$
has a decomposition $\phi\cong\phi_0\perp\ql(\phi)$
with $\phi_0$ nonsingular of dimension $2r$, and $\ql(\phi)$ totally singular of
dimension $s$, so $n=2r+s$.  In such a decomposition, $\ql(\phi)$ is uniquely determined 
up to isometry (it is essentially the restriction of $\phi$ to its radical) and it is called
the quasi-linear part of $\phi$.
In particular, the values $r,s$ are uniquely determined.
However, $\phi_0$ is generally not uniquely determined up to isometry if $s>0$.
The pair $(r,s)$ is called the type of $\phi$, and we say that $\phi$ is
\begin{itemize}
	\item nonsingular if $s=0$;
	\item singular if $s>0$;
	\item semi-singular if $r>0$ and $s>0$;
	\item totally singular if $r=0$ and $s>0$;
	\item nondefective if $\ql(\phi)$ is anisotropic;
	\item nondegenerate if $s\leq 1$ and $\ql(\phi)$ is anisotropic.
\end{itemize}
The above decomposition can be refined.  There are unique nonnegative integers
$i,j$ and a uniquely determined (up to isometry) anisotropic form $\phi_\an$
(called the anisotropic part of $\phi$) such that
$$\phi\cong i\times \HH \perp \phi_\an\perp j\times\qf{0}.$$
One defines the Witt index (resp. the defect) of $\phi$ to be
$\iw(\phi)=i$ (resp. $\iql(\phi)=j$).  We then have $\itt(\phi)=\iw(\phi)+\iql(\phi)$
and $\phi$ being nondefective just means that $\iql(\phi)=0$.
The nondefective part of $\phi$ is given by $i\times\HH\perp\phi_\an$ and
it is also uniquely determined up to isometry.  We call $\phi$ and $\psi$ 
Witt-equivalent, denoted by $\phi\wit\psi$, if $\phi_\an\cong\psi_\an$.

\begin{remark}\label{qperpq}
Using the relations $[a,b]\perp [a,b]\cong 2\times\HH$ and $\qf{a,a}\cong\qf{a,0}$,
it is now easy to see that with $\phi\cong\phi_0\perp\ql(\phi)$
as above, we get $\phi\perp\phi\wit\ql(\phi)$.  We will use this later on.
\end{remark}

Two forms $\phi$ and $\psi$ are similar if there exists an $a\in F^*$ such that
$\phi\cong a\psi$, in which case we write $\phi\sml\psi$.
If $E/F$ is a field extension, then $\phi_E=\phi\otimes E$ is the
form over $E$ obtained by scalar extension.

We say that the two forms $\phi$ and $\psi$ with $\dim\phi=\dim\psi$
are Vishik-equivalent, denoted by
$\phi\vis\psi$, if, over every field
extension $E/F$, we have 
$$\iw(\phi_E)=\iw(\psi_E)\quad\mbox{and}\quad \iql(\phi_E)=\iql(\psi_E).$$
It is known that if $\phi\vis\psi$, then the types of $\phi$ and
$\psi$ are the same, and that if, say, $\phi$ is nondefective, then
$\phi\vis\psi$ iff $\itt(\phi_E)=\itt(\psi_E)$ for all field extensions $E/F$
(\cite[Proposition 2.5, Lemma 2.8]{hz}).

 We say that $\psi$ is dominated by $\phi$, denoted by $\psi\dom\varphi$,
 if for some subspace $U$ of the underlying vector space $V$ of $\phi$,
 one has $\varphi|_U\cong\psi$.   In this situation, we put
 $U^\perp=\{ x\in V\,|\,b_\phi(x,y)=0\ \forall y\in U\}$ and we call a form
 $\eta$ a  
 complement of $\psi$ in $\phi$ if $\eta\cong \phi|_{U^\perp}$.
 Such a complement is unique up to isometry and will also be denoted by
 $\psi^\com_\phi$. 
 
 The following lemma is folklore, we include a proof for the reader's convenience.
 \begin{lemma}\label{dom-isotropic}
 Let $\phi$ and $\psi$ be forms over $F$ with $\psi\dom\phi$.
 If $\dim\psi+\itt(\phi)>\dim\phi$, then $\psi$ is isotropic.
 \end{lemma}
 \begin{proof}  We may assume that $\phi$ is defined on the vector space $V$ and $\psi=\phi|_U$
 is the restriction of $\phi$ to a subspace $U\subset V$.  Let $W$ be a maximal totally
 isotropic subspace of $\phi$, so in particular $\dim W=\itt(\phi)$.  Then
 $$\dim U+\dim W=\dim\psi+\itt(\phi)>\dim\phi =\dim V,$$
 and therefore, $U$ and $W$ intersect nontrivially.  Thus, there exists some
 $0\neq x\in U\cap W$.  In particular, $\psi(x)=\phi(x)=0$
 and $\psi$ is isotropic.
 \end{proof}   
 
 Note that if $\phi$ is nonsingular and $\psi\dom\phi$, then
 $\dim\psi + \dim\psi^\com_\phi =\dim\phi$.  More precisely, we have the following
 (see \cite[Lemma 3.1, Remark 3.2, Lemma 3.7]{hl}):
 \begin{lemma}\label{complement}
 Let $\psi$ be a form over $F$ of type $(r,s)$ and write 
 $\psi\cong\psi_0\perp\qf{a_1,\ldots,a_s}$ with $\psi_0$
 nonsingular of dimension $2r$ and $a_i\in F$.
 Let $\phi$ be a nonsingular form over $F$ with $\psi\dom\phi$.
 Then there exist $c_1,\ldots,c_s\in F$ and a nonsingular form $\psi_1$ over $F$
 such that
 $$\phi\cong\psi_0\perp [a_1,c_1]\perp\ldots\perp [a_s,c_s]\perp \psi_1.$$
 In this situation, we have $\psi^\com_\phi\cong \psi_1\perp\qf{a_1,\ldots,a_s}$.
 In particular,
 $$\phi\perp\psi\cong (s+\dim\psi_0)\times\HH\perp\psi^\com_\phi$$
 and thus $\phi\perp\psi\wit\psi^\com_\phi$.
 Furthermore,
 $\typ(\psi^\com_\phi)=(\frac{\dim\phi}{2}-r-s,s)$ and (by symmetry)
 $$\bigl(\psi^\com_\phi\bigr)^\com_\phi\cong\psi.$$
 \end{lemma}
 
(To get the Witt equivalence, note that $[a_i,c_i]\perp\qf{a_i}\cong\HH\perp\qf{a_i}$.) 
In the situation of the above lemma, we say that $\varphi$ is a
 \emph{nonsingular completion} of $\psi$ if 
 $\phi$ of type $(r+s,0)$, i.e., $\dim\psi_1=0$.  So the nonsingular
 completions of $\psi$ will be exactly the forms
 of shape $\psi_0\perp [a_1,c_1]\perp\ldots\perp [a_s,c_s]$
 for any choice of the $c_i$'s. Choosing all $c_i=0$, we see for example,
 that $\psi_0\perp s\times\HH$ is a nonsingular completion of $\psi$.
 
\section{Pfister forms and half-neighbors}
Here, a bilinear form is always understood to be symmetric
and we still assume that $F$ is a field of characteristic $2$.
Let $b:V\times V\to F$ be a bilinear form defined on an
$F$-vector space $V$ with $\dim V=n<\infty$, and let $\gr_\BB(b)\in M_n(F)$
be the Gram matrix of $b$
with respect to a basis $\BB$ of $V$. We say that $\BB$ is an orthogonal basis
if $\gr_\BB(b)$ is a diagonal matrix, say, 
$\gr_\BB(b)=\diag(a_1,\ldots, a_n)$.  In this case, we also write
$b\cong\qf{a_1,\ldots,a_n}_b$. $b$ is said to be nondegenerate
iff $\det\gr_\BB(b)\neq 0$, which is equivalent to
$\rad(b)=\{ x\in V\,|\,b(x,y)=0\ \forall y\in V\}=0$.  If $b\cong\qf{a_1,\ldots,a_n}_b$,
then $b$ is nondegenerate if $a_i\in F^*$ for all $1\leq i\leq n$.

We define the $0$-fold bilinear Pfister form
as $\qf{1}_b$, and for an integer $m>0$,  an $m$-fold bilinear Pfister form is a nondegenerate
$2^m$-dimensional bilinear form obtained as an $m$-fold
tensor product $\qf{1,a_1}_b\otimes\ldots\otimes\qf{1,a_m}_b$ with $a_i\in F^*$,
we write $\pff{a_1,\ldots,a_n}_b$ for short.  By multiplying out, we obtain
$$\pff{a_1,\ldots,a_n}_b=\bp_{I\subseteq\{1,\ldots ,n\}}\qf{a_I}_b\quad\mbox{where
$a_I=\prod_{k\in I}a_k$,}$$
for example, $\pff{a_1,a_2}_b\cong\qf{1,a_1,a_2,a_1a_2}_b$.

If $b$ is a bilinear form defined on the $F$-vector space $V$, and $q$ is a quadratic form
defined on the $F$-vector space $W$, then one obtains a quadratic form $b\otimes q$
defined on $V\otimes_FW$ uniquely determined by the property that
$b\otimes q(x\otimes y)=b(x,x)q(y)$ for all $x\in V$, $y\in W$.  If
$b\cong\qf{a_1,\ldots,a_n}_b$ is diagonal, then
$$b\otimes q\cong a_1q\perp \ldots\perp a_nq$$.

For an integer $m\geq 0$, an $(m+1)$-fold quadratic Pfister form (or $(m+1)$-Pfister for short)
is a tensor product of an $m$-fold bilinear Pfister form $\pff{a_1,\ldots,a_m}_b$, $a_i\in F^*$,
and a binary quadratic form $[1,c]$, $c\in F$, which we denote by
$$\pff{a_1,\ldots,a_m}_b\otimes [1,c]=\qpf{a_1,\ldots,a_m,c}\cong 
\bp_{I\subseteq\{1,\ldots ,n\}}a_I[1,c].$$
An $(m+1)$-Pfister $\pi$ is nonsingular and has the following important properties:
\begin{itemize}
\item $\pi$ is isotropic $\implylr$ $\pi$ is hyperbolic;
\item $G_F(\pi)=D_F(\pi)$, i.e., $\pi$ is round.
\end{itemize} 
The set of all forms over $F$ that are isometric (resp. similar)
to $(m+1)$-Pfisters is denoted by
$P_{m+1}F$ (resp. $GP_{m+1}F$).

Note that $2^m\times\HH\in P_{m+1}F$:  in the above definition of a Pfister form,
just choose any $a_i\in F^*$ and $c=0$, so $[1,c]\cong\HH$.

\begin{definition}  Two forms $\phi$ and $\psi$ over $F$ are said to be 
{\em half-neighbors} (of each other) if there exist an integer $m\geq 0$, a form
$\pi\in P_{m+1}F$, $a,b\in F^*$ such that
\begin{itemize}
\item $\dim\phi=\dim\psi=2^m$, and
\item $a\phi\dom\pi$, and
\item $b\psi\cong (a\phi)^\com_\pi$.
\end{itemize}
We then also say that $\phi$ and $\psi$ are half-neighbors with respect to the
Pfister form $\pi$
\end{definition}

\begin{remark}
(i) Every form $\phi$ of dimension $2^m$ is a half-neighbor of itself.  Indeed, write
$\phi\cong\phi_0\perp \qf{a_1,\ldots,a_s}$ with $\phi_0$ nonsingular of dimension
$2r$ and note that $\phi_0\perp\phi_0\cong 2r\times\HH$.  Thus,
$$\phi\dom \pi=\phi_0\perp [a_1,0]\perp\ldots\perp [a_s,0]\perp\phi_0\cong 2^m\times\HH\in P_{m+1}F$$
and $\phi^\com_\pi\cong \phi\wit\pi\perp\phi$.

\medskip

\noindent (ii)  If $\phi$ and $\psi$ are half-neighbors, the Pfister form
used to obtain this relation may not be unique.  As an example, assume that $\qpf{a,c}$ 
is anisotropic.  Then $\qpf{a,c}\cong [1,c]\perp a[1,c]$, so $[1,c]$ is a half-neighbor
of itself with respect to the anisotropic Pfister form  $\qpf{a,c}$.  But by
the above, $[1,c]$ is also a half-neighbor of itself with respect to the
hyperbolic $2$-Pfister.
\end{remark}

\begin{proposition}\label{hn-prop}
Let $m\geq 0$ be an integer. Let  $\pi\in P_{m+1}F$ and let $\phi$ and $\psi$
be forms over $F$ of dimension $2^m$ such that
$\phi\dom\pi$ and $\psi\cong\phi^\com_\pi$ with $\typ(\phi)=(r,s)$. Then:
\begin{itemize}
\item[(i)] $\typ(\phi)=\typ(\psi)=(r,s)$.
\item[(ii)] If $\pi$ is isotropic then $\phi\cong\psi$.
\item[(iii)] Suppose $\pi$ is anisotropic.  Then $\phi\cong\psi$
iff $r=0$.
\item[(iv)] $\phi\vis\psi$.
\end{itemize}
\end{proposition}

\begin{proof} 
(i) We have $\dim(\phi)=2r+s=2^m=\frac{\dim\pi}{2}$ and by Lemma \ref{complement}
we get 
$$\typ(\psi)=\typ(\phi^\com_\pi)=(\textstyle{\frac{\dim\pi}{2}}-r-s,s)=(r,s)=\typ(\phi).$$

\medskip

\noindent (ii) 
By Lemma \ref{complement}, we have $\pi\perp\phi\wit\psi$.
If $\pi$ is isotropic and hence hyperbolic, we get 
that $\phi\wit\psi$, and since the types are the same, this immediately implies
$\phi\cong\psi$.

\medskip

\noindent(iii)  Write $\phi\cong\phi_0\perp\sigma$ with $\phi_0$ nonsingular of
dimension $2r$, and $\sigma$ totally singular of dimension $s$.
Then, by Lemma \ref{complement},
we can write
$\psi\cong\psi_0\perp\sigma$ with $\psi_0$ nonsingular of
dimension $2r$.  Hence, if $r=0$, we have 
$\phi\cong\sigma\cong\psi$. 

So let $r>0$ and thus $\dim\sigma<2^m$, and suppose $\phi\cong\psi$. By Remark \ref{qperpq}
we have $\phi\perp\phi\wit\sigma$, and with $\pi\perp\phi\wit\psi\cong\phi$, we get
$$\pi\perp\sigma\wit\pi\perp\phi\perp\phi\wit\phi\perp\phi\wit\sigma.$$
Comparing types and Witt indices, we get
$$\itt(\pi\perp\sigma)\geq \iw(\pi\perp\sigma)=2^m>\dim\sigma=\dim(\pi\perp\sigma)-\dim(\pi),$$
hence $\pi$ is isotropic by Lemma \ref{dom-isotropic}, a contradiction.

\medskip

\noindent (iv) (See also \cite[Lemma 6.2]{hz}.)
If $E/F$ is a field extension and $\pi_E$ is anisotropic,
then both $\phi$ and $\psi$ 
are anisotropic as well as $\phi_E,\psi_E\dom \pi$.
If $\pi_E$ is isotropic and hence hyperbolic, then 
by (ii), we have $\phi_E\cong\psi_E$. So in any case, we have
$(\iql(\phi_E),\iw(\phi_E))=(\iql(\psi_E),\iw(\psi_E))$.
\end{proof}

The following corollary contains the basic idea behind one of our constructions of
certain semi-singular Vishik-equivalent forms that are not similar.
\begin{corollary}\label{nonsim-hn}
Let $\phi$ and $\psi$ be semi-singular half-neighbors of dimension $2^m$ with respect to
some anisotropic Pfister form $\pi\in P_{m+1}F$.  
Suppose that $G_F(\ql(\phi))=F^{*2}$.
Then $\phi$ and $\psi$ are Vishik-equivalent
but not similar.
\end{corollary}

\begin{proof} 
The fact that $\phi$ and $\psi$ are Vishik-equivalent has been shown in  
Lemma \ref{hn-prop}(iv).  After scaling, we may assume that
$\phi\dom\pi$ and $\psi\cong \phi^\com_\pi$.

Write $\phi\cong\phi_0\perp\sigma$ with $\phi_0$ nonsingular
of dimension $2r>0$ and $\ql(\phi)=\sigma$ totally singular
of dimension $s>0$.  Then by Lemma \ref{complement}, 
$\psi\cong\phi^\com_\pi\cong\psi_0\perp\sigma$ with $\psi_0$ nonsingular of dimension $2r$.

Suppose $\phi\sml\psi$, then there exists some $\lambda\in F^*$ such that
$$\lambda\phi\cong\lambda\phi_0\perp\lambda\sigma\cong \psi\cong\psi_0\perp\sigma.$$
The uniqueness (up to isometry) of the quasi-linear part forces
$\lambda\sigma\cong\sigma$, so $\lambda\in G_F(\sigma)$.  By assumption,
this implies $\lambda\in F^{*2}$ and thus 
$$\phi\cong\lambda\phi\cong\psi,$$
a contradiction to Proposition \ref{hn-prop}(iii).
\end{proof}

\begin{remark} We already know from the introduction
that Vishik-equivalent forms of dimension $\leq 5$ are similar, so to be able
to construct a situation as in this corollary, we must have $m\geq 3$. 

A different way of seeing this is as follows. First note
that semi-singularity implies that $\dim\phi=2r+s=2^m$ with $r,s>0$, and therefore
$m\geq 2$.  Now if $m=2$, we would have $r=1$ and $s=2$, so, after scaling
(which doesn't affect the group of similarity factors) we may assume
that $\ql(\phi)\cong\qf{1,a}$ with $a\in F^*$.  The anisotropy implies
$a\notin F^{*2}$.  In the next section, we will see (Lemma  \ref{totsing-sim}) that this implies
$F^{*2}\subsetneq G_F(\sigma)=F^2(a)^*$ (where $F^2(a)$ is the field extension
generated by $a$ over the field $F^2$ inside $F$).
Thus, to have a situation as in the above corollary with 
$G_F(\sigma)=F^{*2}$, we must have $m\geq 3$ and $s\geq 4$.
\end{remark}

\section{Totally singular forms and their similarity factors}
Recall from Section \ref{Sec:Def} that a totally singular form $\sigma$ over
$F$ has a decomposition $\sigma\cong\sigma_{\an}\perp i\times\qf{0}$ where
$\sigma_{\an}$ is anisotropic, totally singular and uniquely determined up to isometry, and
$i=\iql(\sigma)$ is the defect. 
Now suppose $\dim\sigma=n\geq 1$ and write $\sigma\cong\qf{a_1,\ldots, a_n}$.
Then $D_F^0(\sigma)=\sum_{i=1}^nF^2a_i$ is a finite-dimensional $F^2$-subspace
inside the $F^2$-vector space $F$.  In fact, there is a one-one correspondence
$$\begin{array}{rcl}
\left\{ \begin{array}{l}\mbox{isometry classes of anisotropic}\\
\mbox{totally singular forms over $F$}\end{array}\right\}
& \longrightarrow & \left\{\begin{array}{l}\mbox{finite-dimensional}\\
\mbox{$F^2$-subspaces of $F$}\end{array}\right\}\\[1ex]
\sigma & \longmapsto & D_F^0(\sigma)
\end{array}$$
(see \cite[Proposition~8.1]{hl}, \cite[p.~56]{ekm}).
If $\phi\cong\qf{a_1,\ldots,a_n}$ and $\dim\phi_{\an}=m\leq n$, then one can find
indices $1\leq i_1<i_2<\ldots<i_m\leq n$ such that
$\phi_{\an}\cong\qf{a_{i_1},\ldots, a_{i_m}}$ by choosing a basis among the $a_i$'s
of the $F^2$-vector space $D_F^0(\phi)$ that is spanned by $a_1,\ldots,a_n$.
In particular, $\dim\phi_{\an}=\dim_{F^2}D_F^0(\phi)$.
With these observations, the following is now rather obvious.
\begin{lemma}\label{totsing-ani}
Let $\sigma$ and $\tau$ be totally singular forms over $F$.  Then there exist $\tau'\dom\tau$
and $\sigma'\dom\sigma$ 
such that
$$(\sigma\perp\tau)_{\an}\cong \sigma_{\an}\perp\tau'\cong\sigma'\perp\tau_{an}.$$
In particular,
$\dim(\sigma\perp\tau)_{\an}\geq\max\{\dim\sigma_{\an},\dim\tau_{\an}\}$.

Furthermore, $\sigma\perp\sigma\wit\sigma$.
\end{lemma}

Let $n\geq 1$ be an integer.  An $n$-fold quasi-Pfister form (or $n$-quasi-Pfister for short)
is a totally
singular form isometric to 
$$\pff{a_1,a_2,\ldots,a_n}:=\qf{1,a_1}\otimes\qf{1,a_2}\otimes\ldots\otimes\qf{1,a_n}$$
for some $a_i\in F$.  We define the $0$-quasi-Pfister to be the form $\qf{1}$.
We have that
$D_F^0(\pff{a_1,a_2,\ldots,a_n})=F^2(a_1,\ldots,a_n)$ is a finite field extension of $F^2$
inside $F$.  Indeed, there is a one-one correspondence
$$\begin{array}{rcl}
\left\{ \begin{array}{l}\mbox{isometry classes of anisotropic}\\
\mbox{quasi-Pfisters over $F$}\end{array}\right\}
& \longrightarrow & 
\left\{ \begin{array}{l}\mbox{finite extensions of}\\
\mbox{$F^2$ inside $F$}\end{array}\right\}\\[1ex]
\pi & \longmapsto & D_F^0(\pi)
\end{array}$$ 
(see \cite[Proposition~8.5]{hl},\cite[Remark~10.4]{ekm}).
If $\pi\cong\pff{a_1,\ldots,a_n}$, then one can find
indices $1\leq i_1<i_2<\ldots<i_m\leq n$ such that
$\pi_{\an}\cong\pff{a_{i_1},\ldots, a_{i_m}}$ by choosing a minimal subset of
the $a_i$'s so that 
$D_F^0(\phi)=F^2(a_1,\ldots,a_n)=F^2(a_{i_1},\ldots, a_{i_m})$.

Important invariants attached to a nonzero totally singular form $\sigma$ over $F$
are the so-called
{\em norm field} $N_F(\sigma)$ of $\sigma$ over $F$ and its norm degree $\ndeg_F(\sigma)$
defined as follows:
$$\begin{array}{rcl}
N_F(\sigma) & = & F^2\bigl(\mbox{$\frac{a}{b}$}\,|\,a,b\in D_F(\sigma)\bigr),\\[1ex]
\ndeg_F(\sigma) & = & [N_F(\sigma):F^2].\end{array}$$
Note that for any $a\in D_F(\sigma)$, we have $aD_F(\sigma)\subset N_F(\sigma)$.
If $\sigma\cong\qf{a_1,\ldots, a_n}$ with, say, $a_1\neq 0$, then we get
$$\begin{array}{rcl} 
N_F(\sigma) & = & F^2\bigl(\mbox{$\frac{a_i}{a_1}$}\,|\,2\leq i\leq n\}\bigr),\\[1ex]
\ndeg_F(\sigma) & = & 2^m\quad \mbox{for some $m\leq n-1$.}\end{array}$$
Next, we collect some facts about the group of similarity factors of a totally singular form
(see \cite[Proposition 8.5]{hl}, \cite[Proposition 6.4]{h}).

\begin{lemma}\label{totsing-sim}
Let $\sigma$ be a nonzero totally singular form over $F$.
\begin{enumerate}
\item[(i)]  $G_F^0(\sigma)$ is a field extension of $F^2$ inside $N_F(\sigma)$.
\item[(ii)]  Let $c_1,\ldots, c_r\in F^*$ such that
$[F^2(c_1,\ldots, c_r):F^2]=2^r$.   Then $F^2(c_1,\ldots, c_r) \subseteq G_F^0(\sigma)$
iff there exists $\sigma'\dom\sigma$ with
$\sigma_{\an}\cong\pff{c_1,\ldots, c_r}\otimes\sigma'$.  In this situation,
$2^r$ divides $\dim\sigma_{\an}$.
\item[(iii)]  Suppose $\sigma$ is anisotropic.  Then $\sigma$ is a quasi-Pfister iff 
$G_F^0(\sigma)=D_F^0(\sigma)$ iff $D_F^0(\sigma)=N_F(\sigma)$.
\item[(iv)] Suppose $\sigma$ is anisotropic.  Then $\sigma$ is similar to a quasi-Pfister
iff $\dim\sigma=\ndeg_F(\sigma)$.
\end{enumerate}
\end{lemma}

\begin{corollary}\label{totsing-sim-even}
Let $\sigma$ be a nonzero totally singular form over $F$.  If $\dim\sigma_{\an}$ is
odd, then $G_F(\sigma)=F^{*2}$.
\end{corollary}
\begin{proof}  Let $\sigma$ be a nonzero totally singular form over $F$ and
suppose there exists some $a\in  G_F(\sigma)\setminus F^{*2}$.  So in particular,
$[F^2(a):F^2]=2$.
Since $G_F^0(\sigma)$ is a field extension  of $F^2$, we have $F^2(a)\subset G_F^0(\sigma)$.
By Lemma \ref{totsing-sim}(ii), there exists $\sigma'\dom\sigma$ such that
$\sigma_{\an}\cong\pff{a}\otimes\sigma'$ and thus $\dim\sigma_{\an}=2\dim\sigma'$ is even.
\end{proof}

As a further consequence, one can show that the group of similarity factors of the quasi-linear part of a semi-singular form defines an invariant of the Vishik-equivalence class of that form.  
\begin{corollary}\label{sim-factors-inv}
Let $\phi$ and $\psi$ be semi-singular or totally singular forms over $F$ with $\phi\vis\psi$.  Then
$G_F(\ql(\phi))=G_F(\ql(\psi))$.
\end{corollary}
\begin{proof}  It follows immediately from the definition of Vishik-equiva\-lence that
$\phi\vis\psi$ implies $\ql(\phi)\vis\ql(\psi)$. 
Thus, to prove the claim, it suffices to assume that $\phi$ and $\psi$ are totally singular
of the same positive dimension.  By Vishik equivalence (applied over $F$ itself), 
$\dim\phi_\an=\dim\psi_\an$.  Now $G_F(\qf{0,\ldots,0})=F^*$, and if
$\dim\phi_\an=\dim\psi_\an>0$, we clearly have $G_F(\phi_\an)=G_F(\phi)$ and 
$G_F(\psi_\an)=G_F(\psi)$.  So without loss of generality, $\phi$ and $\psi$ are anisotropic.

Let $a\in G_F(\phi)$.  If
$a\in F^{*2}$ then clearly $a\in G_F(\psi)$.  So suppose that
$a\in G_F(\phi)\setminus F^{*2}$.
Then, by Lemma~\ref{totsing-sim},
there is a totally singular form $\sigma_1$ with $\phi\cong\pff{a}\otimes\sigma_1$.
(Note that necessarily, $\dim\ql(\phi)$ is even.) 

Now consider $E=F(\sqrt{a})$.  It follows from \cite[Corollary 5.8(iii)]{h} that
$\iql(\phi_E)=\frac{1}{2}\dim\phi$. (In fact, we have
$(\phi_E)_\an\cong\sigma_1$.) By Vishik-equivalence, this also forces
$\iql(\psi_E)=\frac{1}{2}\dim\psi$ which, once more by
\cite[Corollary 5.8(iii)]{h}, implies that there exists a totally singular form
$\sigma_2$ with $\psi\cong\pff{a}\otimes\sigma_2$.  This in turn implies that
$a\in  G_F(\psi)$ and thus $G_F(\phi)\subset G_F(\psi)$.
The reverse inclusion follows by symmetry.
\end{proof}

\begin{proposition}\label{totsing-sim-trivial}
Let $\pi$ be an anisotropic quasi-Pfister and let $\sigma$ and $\tau$ be totally
singular forms with $\dim\sigma=s\geq 1$, $\dim\tau=t\geq 1$  and such that
$\sigma\dom\pi$, $\tau\dom\pi$.  Let $x\in F^*\setminus D_F(\pi)$.
\begin{enumerate}
\item[(i)] $\pff{x}\otimes\pi\cong\pi\perp x\pi$ is anisotropic and therefore
also $\sigma\perp x\tau$.
\item[(ii)]  If $s\neq t$, then
$G_F(\sigma\perp x\tau)=G_F(\sigma)\cap G_F(\tau)$.
\item[(iii)]  Suppose $s\neq t$.  If $\dim\sigma$ is odd or
if $\dim\tau$ is odd, then $G_F(\sigma\perp x\tau)\cong F^{*2}$.
\end{enumerate}
\end{proposition}

\begin{proof}
(i)  Suppose $\pi\perp x\pi$ is isotropic. Then the anisotropy of the quasi-Pfister
$\pi$ implies that 
there exist $u,v\in D_F(\pi)=G_F(\pi)$ with $u+xv=0$, i.e. $u=xv$, and hence
$x=u/v\in G_F(\pi)=D_F(\pi)$, a contradiction.
Obviously,
$\sigma\perp x\tau\dom\pi\perp x\pi\cong\pi\otimes\pff{x}$ and thus 
$\sigma\perp x\tau$ is anisotropic as well.

(ii)  It is clear that  $G_F(\sigma)\cap G_F(\tau)\subseteq G_F(\sigma\perp x\tau)$
(even without the assumption on the dimensions $s,t$).  Conversely, suppose
$y\in G_F(\sigma\perp x\tau)$.  Let $z\in D_F(\sigma)$.  Then
$$yz\in D_F(y(\sigma\perp x\tau))=D_F(\sigma\perp x\tau)\subseteq
D_F(\pi\otimes\pff{x})=G_F(\pi\otimes\pff{x}),$$
but $z\in  D_F(\sigma)\subseteq D_F(\pi\otimes\pff{x})=G_F(\pi\otimes\pff{x})$ as well,
hence $y\in G_F(\pi\otimes\pff{x})$.

If $y\in D_F(\pi)=G_F(\pi)$, then $\sigma,y\sigma\dom\pi$ and hence
$(\sigma\perp y\sigma)_{\an}\dom\pi$, and similarly, 
$(\tau\perp y\tau)_{\an}\dom \pi$.  Hence,
$(\sigma\perp y\sigma)_{\an}\perp x(\tau\perp y\tau)_{\an}$ is anisotropic (similarly
as in (i)).
Since $y\in G_F(\sigma\perp x\tau)$, we obtain
$$\begin{array}{rcl}
\sigma\perp x\tau & \wit &   (\sigma\perp x\tau) \perp (\sigma\perp x\tau)\\
 & \wit & (\sigma\perp x\tau) \perp y(\sigma\perp x\tau)\\
 & \wit & (\sigma\perp y\sigma)\perp x(\tau\perp y\tau)\\
 & \wit & (\sigma\perp y\sigma)_{\an}\perp x(\tau\perp y\tau)_{\an}
 \end{array}$$
 and the anisotropy implies 
 $$\sigma\perp x\tau\cong (\sigma\perp y\sigma)_{\an}\perp x(\tau\perp y\tau)_{\an}.$$
 But by Lemma \ref{totsing-ani}, we have
 $\sigma\dom (\sigma\perp y\sigma)_{\an}$ and $\tau\dom(\tau\perp y\tau)_{\an}$.
 Comparing dimensions, this necessitates
 $\sigma\cong (\sigma\perp y\sigma)_{\an}$ and $\tau\cong(\tau\perp y\tau)_{\an}$
 and thus $D_F^0(y\sigma)\subseteq D_F^0(\sigma)$ and $D_F^0(y\tau)\subseteq D_F^0(\tau)$.
 By the anisotropy of $\sigma$ and $\tau$ and by comparing the $F^2$-dimensions
 of the value sets we get $y\sigma\cong \sigma$ and $y\tau\cong\tau$, so
 $y\in G_F(\sigma)\cap G_F(\tau)$.
 
Suppose $y\notin D_F(\pi)$.  Then  $\pi\otimes\pff{y}$ and therefore also 
$\sigma\perp y\sigma$ and $\tau\perp y\tau$ are anisotropic
(similarly as in (i)).  
By the same reasoning as above, we get
$$\sigma\perp x\tau\wit (\sigma\perp y\sigma)\perp x(\tau\perp y\tau)$$
and by Lemma \ref{totsing-ani} and the anisotropy of 
$\sigma\perp y\sigma$ and $\tau\perp y\tau$ we get
$$\dim\bigl((\sigma\perp y\sigma)\perp x(\tau\perp y\tau)\bigr)_{\an}\geq \max\{ 2s,2t\},$$
but $s\neq t$ and thus $\dim(\sigma\perp x\tau)_{\an}=s+t<\max\{ 2s,2t\}$,
a contradiction.

(iii) follows from (ii) and Corollary \ref{totsing-sim-even}.
\end{proof}

\section{Construction of nonsimilar half-neighbors}
Let $n,r,s$ be nonnegative integers with $n\geq 3$, $r\geq 1$
and $2^n=2r+s$.  Note that $s$ will be even.
To prove the Main Theorem and in view of Proposition \ref{hn-prop}(iv),
it suffices to produce nonsimilar half-neighbors $\phi$ and $\psi$ of some
anisotropic $\pi\in P_{n+1}F$ over a suitably chosen field $F$.

Given any field $F_0$ of characteristic $2$, it was shown in \cite[Theorem~6.5]{hz} that there
exists a field extension $F/F_0$ with an anisotropic Pfister form $\pi\in P_{n+1}F$
and nonsingular half-neighbors $\phi$ and $\psi$ of $\pi$ such that $\phi$ and $\psi$ 
are not similar.  This settles the case of type $(2^{n-1},0)$ and so we focus in our constructions on the remaining cases with $r,s>0$.  

\subsection*{Construction A}
In this construction, we provide examples of nonsimilar half-neighbors of dimension
$2^n$ ($n\geq 3$) and type $(r,2k)$ (so $2r+2k=2^n$) and $2\leq k\leq 2^{n-1}-1$.
This will not yield examples of type $(2^{n-1}-1,2)$, a case covered in our
Construction B.

Let $F_0$ be any field of characteristic $2$ and let $F/F_0$ be any field extension
such that there exist $a_1,\ldots, a_n,x\in F^*$ such that for
the bilinear Pfister form $\pi_b=\pff{a_1,\ldots, a_n}_b$ and its associated
quasi-Pfister $\pi\cong \pff{a_1,\ldots, a_n}$, we have that both 
$$\pi_b\otimes [1,x]\cong \qpf{a_1,\ldots,a_n,x}\quad\mbox{and}\quad\pi\otimes\pff{x}\cong\pff{a_1,\ldots, a_n,x}$$
are anisotropic over $F$.  We may, for example, take a rational function field in
$n+1$ variables (which we call $a_1,\ldots, a_n,x$) over $F_0$.

Write $2k=\ell+m$ with $\ell, m$ odd and $0<\ell<m$
(it is here that we need $k\geq 2$).  Choose any diagonalization of the
{\em bilinear} Pfister form $\pi_b$ of the following type:
$$\pi_b\cong\qf{c_1,\ldots, c_r,d_1,\ldots, d_\ell, e_1,\ldots,e_m,f_1,\ldots f_r}_b$$
for suitable $c_i, d_i,e_i,f_i\in F^*$ (note that $r+\ell+m+r=2r+2k=2^n$).
So we get 
$$\begin{array}{rcl}
\pi_b\otimes [1,x] & \cong & \qf{c_1,\ldots, c_r}_b\otimes [1,x]\perp \qf{d_1,\ldots, d_\ell}_b\otimes [1,x]\\
 & & \perp  \qf{e_1,\ldots,e_m}_b\otimes [1,x]\perp \qf{f_1,\ldots f_r}_b\otimes [1,x].
 \end{array}$$
We have 
$$\begin{array}{rcl}
\delta:=\qf{d_1,\ldots, d_\ell}\dom\pi & \mbox{and} & \qf{d_1,\ldots, d_\ell}\dom
\qf{d_1,\ldots, d_\ell}_b\otimes [1,x]\\
\epsilon:=\qf{e_1,\ldots,e_m}\dom\pi & \mbox{and} & x\qf{e_1,\ldots,e_m}\dom\qf{e_1,\ldots,e_m}_b\otimes [1,x].\end{array}$$
Now consider
$$\phi:= \qf{c_1,\ldots, c_r}_b\otimes [1,x]\perp\delta\perp x\epsilon\dom \pi_b\otimes [1,x].$$
Then $\dim\phi =2r+\ell+m=2r+2k=2^n$, and we see by Lemma \ref{complement} that 
$$\psi:=\qf{f_1,\ldots f_r}_b\otimes [1,x]\perp \delta\perp x\epsilon$$ is the complement
of $\phi$ in $\pi_b\otimes [1,x]$.  Hence, $\phi$ and $\psi$ are half-neighbors of each other
with respect to the anisotropic Pfister form $\pi_b\otimes [1,x]$.

Note that $\ql(\phi)\cong \delta\perp x\epsilon$ with $\delta,\epsilon\dom\pi$ and
$\pi\otimes\pff{x}$ anisotropic, so in particular, $x\notin D_F(\pi)$.
Furthermore,  $\dim\delta=\ell <m=\dim\epsilon$ and $\ell,m$ odd.
By Proposition \ref{totsing-sim-trivial}(iii) we have that $G_F(\ql(\phi))=F^{*2}$,
and by Corollary \ref{nonsim-hn}, $\phi$ and $\psi$ are nonsimilar half-neighbors.

\subsection*{Construction B}
In this construction, we provide examples of nonsimilar half-neighbors of dimension
$2^n$ ($n\geq 3$) and type $(r,2^m)$ and $1\leq m\leq n-2$.  We then have
$r=2^{n-1}-2^{m-1}$.  For $m=1$, this then will also cover the remaining case of type $(2^{n-1}-1,2)$.
We use the same setting with $F/F_0$, $a_1,\ldots,a_n,x$, $\pi_b$ and $\pi$
as in Construction A (but we actually do not require $\pff{a_1,\ldots,a_n,x}$ to be 
anisotropic).

Since $1\leq m\leq n-2$, we get an integer $k:=2^{n-1}-2^{m}-2^{m-1}>0$
and we can diagonalize $\pi_b$
in the following way for suitably chosen $c_i,d_i\in F^*$ and with
the bilinear $m$-Pfister $\alpha_b=\pff{a_1,\ldots,a_m}_b$:
$$\pi_b\cong\alpha_b\perp a_n\alpha_b\perp
\underbrace{\qf{c_1,\ldots,c_k}_b}_{\gamma_b}
\perp\underbrace{\qf{d_1,\ldots,d_\ell}_b}_{\delta_b}$$
where $\ell=2^n-2\dim\alpha_b-k=2^{n-1}-2^{m-1}>0$.

Let $\alpha\cong\pff{a_1,\ldots,a_m}$ be the totally singular quasi-Pfister associated with
$\alpha_b$ and define
$$\begin{array}{rcl}
\phi & \cong & \alpha_b\otimes [1,x]\perp \gamma_b\otimes[1,x]\perp a_n\alpha\\
 & \dom &  \alpha_b\otimes [1,x]\perp \gamma_b\otimes[1,x]\perp a_n\alpha_b\otimes [1,x]
 \perp\delta_b\otimes [1,x]\cong\pi_b\otimes [1,x]
\end{array}$$
with
$$\dim\phi=3\dim\alpha_b+2k=2^n=\frac{1}{2}\dim\pi_b\otimes [1,x].$$
The complement of $\phi$ in $\pi_b\otimes [1,x]$ is given by the form
$$\psi=\delta_b\otimes [1,x]\perp a_n\alpha$$
of dimension $2\ell +2^m=2^n$.
In particular, $\phi$ and $\psi$ are half-neighbors of each other with respect to 
$\pi_b\otimes [1,x]$.

Suppose $\phi\sml\psi$.  Then there exists some $\lambda\in F^*$ such that
$\lambda\phi\cong\psi$ which in turn implies $\lambda\ql(\phi)\cong\ql(\psi)$ by the uniqueness
of the quasi-linear part (up to isometry).
But $\ql(\phi)\cong\ql(\psi)\cong a_n\alpha$, so
$$\lambda\in G_F(a_n\alpha)=G_F(\alpha)=D_F(\alpha)=D_F(\alpha_b)=G_F(\alpha_b)$$
by the roundness of bilinear resp. quasi-Pfister forms.  This implies
$\lambda\alpha_b\otimes [1,x]\cong \alpha_b\otimes [1,x]$ and we get
$$\lambda\phi\cong \alpha_b\otimes [1,x]\perp \lambda\gamma_b\otimes[1,x]\perp a_n\alpha.$$
Since 
$\psi\perp\psi\wit a_n\alpha$ (see Remark \ref{qperpq}), we obtain
$$\begin{array}{rcl}
\underbrace{\lambda\gamma_b\otimes [1,x] \perp a_n\alpha}_A
& \wit & \lambda\gamma_b\otimes [1,x] \perp \psi\perp\psi\\[-1.5ex]
& \wit & \lambda\gamma_b\otimes [1,x] \perp \lambda\phi\perp\psi\\[1ex]
& \wit & \lambda\gamma_b\otimes [1,x]\perp
\underbrace{\alpha_b\otimes [1,x]\perp\lambda\gamma_b\otimes[1,x]\perp a_n\alpha}_{\lambda\phi}\\[-.5ex]
 & &
\perp \underbrace{\delta_b\otimes [1,x]\perp a_n\alpha}_\psi\\
& \wit & \underbrace{\alpha_b\otimes [1,x]\perp\delta_b\otimes [1,x]\perp a_n\alpha}_B\dom\pi_b\otimes [1,x].
\end{array}$$
We see that the form $B$ is anisotropic as it is dominated by the aniso\-tropic
Pfister form $\pi_b\otimes [1,x]$.  Also,
$$\begin{array}{rcl} \dim A & = & 2k+2^m= 2^n-2^{m+1}-2^m+2^m= 2^n-2^{m+1},\\[1ex]
\dim B & = & 2\cdot 2^m+2\ell+2^m=2^n+2^{m+1}.\end{array}$$
Thus, $\dim B>\dim A$, a contradiction to Witt equivalence and the anisotropy of $B$.
Hence, $\phi$ and $\psi$ are nonsimilar half-neighbors as desired.

\begin{remark}  In Construction A, we produced half-neighbors whose quasi-linear part $\sigma$
has the  `smallest' possible group of similarity factors:  $G_F^0(\sigma)=F^2$, i.e.,
$[G_F^0(\sigma):F^2]=1$.  In
Construction B, the group of similarity factors of $\sigma$ is in some sense
as large as possible since $\sigma$ is similar to an $m$-fold quasi-Pfister
and therefore,  $[G_F^0(\sigma):F^2]=2^m=\dim\sigma$ (see Lemma \ref{totsing-sim}).
\end{remark}

\begin{remark}  It seems tempting to apply the half-neighbor construction also in the 
case of totally singular forms using quasi-Pfister forms.
Say, two totally singular forms $\phi$ and $\psi$ of dimension $2^n$ are called
quasi-half-neighbors if there exist $a,b\in F^*$ such that $a\phi\perp b\psi\cong\pi$
for some  $(n+1)$-fold quasi-Pfister form $\pi$.  Then in general, we won't even have
that $\phi$ and $\psi$ are Vishik-equivalent.  For example,
$\qf{1}\perp\qf{0}\cong \pff{1}$, so $\qf{1}$ and $\qf{0}$ would be quasi-half-neighbors.

Even if we require $\pi$ to be anisotropic, we won't get Vishik-equiva\-lence in general
for $n\geq 1$.  Let $n=1$ (for higher $n$, we get analogous counterexamples) and take any field
$F$ with an anisotropic 
$2$-fold quasi-Pfister $\pi=\pff{a,b}$, let $\phi\cong\qf{1,a}$, $\psi\cong\qf{1+b,ab}$.
Then $\phi\perp\psi\cong\pi$ since $\qf{1,1+b}\cong\qf{1,b}$.  Let $E=F(\sqrt{a})$.  Then $(\phi_E)_\an\cong\qf{1}$ whereas
$\psi_E\cong\qf{1+b,b}\cong\qf{1,b}$ stays anisotropic.  Note that in this example,
$G_F^0(\phi)=F^2(a)\neq F^2(ab(1+b))=G_F^0(\psi)$, also implying that the two forms
are not Vishik-equivalent (see Corollary \ref{sim-factors-inv}).
\end{remark}

\begin{remark}  We know (see the Introduction) that Vishik-equiva\-lence implies
similarity in dimensions $\leq 5$ and 
and that there are counterexamples in dimension $8$
for all types that are not totally singular.  It would be nice to close the gap in dimensions
$6$ and $7$.  Further cases where Vishik-equivalence implies similarity in dimension $6$ or $7$
include the cases where one of the forms is as follows:
\begin{itemize}
\item dimension $6,7$: Pfister neighbor (see Introduction);
\item dimension $6$:  Albert form (i.e., $6$-dimensional nonsingular form with trivial
Arf-invariant), \cite[Th\'eor\`eme~1.1]{l};
\item dimension $6$: the form contains (up to a scalar) a $2$-fold quadratic Pfister form as a subform but is not a Pfister neighbor.  This case can be deduced from \cite[Proposition~6.1.12]{f};
\item dimension $7$:  type $(r,s)\in \{ (3,1),(2,3)\}$ (see Introduction).
\end{itemize}
In all other cases of nonsingular or semi-singular forms in dimensions $6$ or $7$, the question
whether Vishik-equivalence always implies similarity or whether there are conterexamples seems
to be open at this point.  In particular, the only open case in
dimension $7$ is that of type $(1,5)$.

Little is known in higher dimensions that are not $2$-powers except for the cases mentioned in
the Introduction.
\end{remark}

\end{document}